\documentclass[12pt,a4paper]{amsart}
    
\usepackage{fullpage}
\makeatletter

\usepackage[english]{babel}
\usepackage[utf8]{inputenc}
\usepackage[square,sort,comma,numbers]{natbib}
\usepackage{graphicx}
\usepackage{amsfonts}
\usepackage{amsthm}
\usepackage{amsmath,amssymb}
\usepackage{xparse}
\usepackage{tikz-cd}
\usepackage{verbatim}
\usepackage{mathtools}
\setcitestyle{square}

\title{Topological Equivalences of E-infinity Differential Graded Algebras}
\author{Haldun Özgür Bayındır}

\NewDocumentCommand{\tens}{t_}
 {%
  \IfBooleanTF{#1}
   {\tensop}
   {\otimes}%
 }
\NewDocumentCommand{\tensop}{m}
 {%
  \mathbin{\mathop{\otimes}\displaylimits_{#1}}%
 }

\newcommand*{\rom}[1]{\expandafter\@slowromancap\romannumeral #1@}

\newcommand{\Z}{\mathbb{Z}}
\newcommand{\Q}{\mathbb{Q}}
\newcommand{\F}{\mathbb{F}}
\newcommand{\Pp}{\mathbb{P}}

\newcommand{\Sp}{\mathbb{S}}
\mathchardef\mhyphen="2D
\newtheorem{theorem}{Theorem}[section]
\newtheorem{lemma}[theorem]{Lemma}

\newtheorem{corollary}[theorem]{Corollary}
\newtheorem{proposition}[theorem]{Proposition}

\theoremstyle{definition}

\newtheorem{definition}[theorem]{Definition}
\newtheorem{example}[theorem]{Example}
\newtheorem{remark}[theorem]{Remark}

\def\co{\colon\thinspace}

\date{}

\usepackage{natbib}
\usepackage{graphicx}

\begin{document}

\maketitle

\begin{abstract}
Two DGAs are said to be topologically equivalent when the corresponding Eilenberg--Mac Lane ring spectra are weakly equivalent as ring spectra. Quasi-isomorphic DGAs are  topologically equivalent but the converse is not necessarily true. As a counter-example, Dugger and Shipley showed that there are DGAs that are  non-trivially topologically equivalent, ie topologically equivalent but not quasi-isomorphic. 

In this work, we define $E_\infty$ topological equivalences and utilize the obstruction theories developed by Goerss, Hopkins and Miller to construct first examples of non-trivially $E_\infty$ topologically equivalent $E_\infty$ DGAs. Also, we show using these obstruction theories that for co-connective $E_\infty$ $\F_p$--DGAs, $E_\infty$ topological equivalences and quasi-isomorphisms agree. For $E_\infty$ $\F_p$--DGAs with trivial first homology, we show that an $E_\infty$ topological equivalence induces an isomorphism in homology that preserves the Dyer--Lashof operations and therefore induces an $H_\infty
$ $\F_p$--equivalence. 
\end{abstract}

\section{Introduction}
Dugger and Shipley defined a new equivalence relation between associative differential graded algebras (which we call DGAs) that they call topological equivalences  \cite{dugger2007topological}. To define topological equivalences, they use the Quillen equivalence between $R$--DGAs and $HR$--algebras where $R$ denotes a discrete commutative ring, see Shipley \cite{shipley2007hz}. Two $R$--DGAs $X$ and $Y$ are said to be \textbf{topologically equivalent} if the corresponding $HR$--algebras $HX$ and $HY$ are weakly equivalent as $\Sp$--algebras where $\Sp$ denotes the sphere spectrum. Using Quillen equivalences in \cite{shipley2007hz}, it is easy to see that topologically equivalent  DGAs are Morita equivalent. Furthermore, topological equivalences appear in one of the equivalent definitions of Morita equivalences of   DGAs, see Theorem 1.4 of \cite{dugger2007topological}.

By the Quillen equivalence between $R$--DGAs and $HR$--algebras, two $R$--DGAs are quasi isomorphic if and only if the corresponding $HR$--algebras are weakly equivalent as $HR$--algebras. Because the forgetful functor from   $HR$--algebras to    $\Sp$--algebras preserves weak equivalences, it is clear that quasi-isomorphic   DGAs are always topologically equivalent. One of the main results of  \cite{dugger2007topological} is that there are   DGAs that are not quasi-isomorphic but are    topologically equivalent. Such    DGAs are called non-trivially   topologically equivalent. On the other hand, another theorem in \cite{dugger2007topological} states that there are no examples of non-trivial   topological equivalences in $\mathbb{Q}$--DGAs, ie topologically equivalent $\Q$--DGAs are quasi-isomorphic. See Theorem \ref{theorem Q} below. 

Because there is also a Quillen equivalence between $E_\infty$ $R$--DGAs and commutative $HR$--algebras, see Richter and Shipley \cite{richter2014algebraic}, topological equivalences for $E_\infty$ DGAs can also be considered. Now we explain what we mean by topological equivalences for DGAs and $E_\infty$ DGAs. For DGAs we have the following definition of topological equivalence.  

\begin{definition}
Two   $R$--DGAs $X$ and $Y$ are   \textit{topologically equivalent} if the corresponding   $HR$--algebras $HX$ and $HY$ are weakly equivalent as $\Sp$--algebras. This is same as the definition of topological equivalence  in \cite{dugger2007topological}.
\end{definition}

\noindent
The definition for topological equivalence of $E_\infty$ DGAs is the following.

\begin{definition}
Two $E_\infty$ $R$--DGAs $X$ and $Y$ are \textit{$E_\infty$ topologically equivalent} if the corresponding commutative $HR$--algebras $HX$ and $HY$ are weakly equivalent as commutative $\Sp$--algebras. 
\end{definition}

Our methods make use of obstruction theories for ring spectra. In \cite{robinson1989obstruction}, Robinson develops an obstruction theory for showing existence of ring structures on spectra. This obstruction theory is generalized for commutative ring spectra by Robinson in  \cite{robinson2003gamma}. Based on the obstruction spectral sequence of Bousfield \cite{bousfield1989homotopy}, Hopkins and Miller developed another obstruction theory \cite{rezk1998notes}. Their obstruction theory provides an obstruction spectral sequence for calculating mapping spaces of   ring spectra and also an obstruction theory for showing the existence of ring  structures on spectra. This obstruction theory is generalized to commutative ring spectra by Goerss and Hopkins \cite{goerss2004moduli}. In this work, we use the $T$--algebra spectral sequence of Johnson and Noel \cite{johnson2014lifting}, which is a generalization of the obstruction spectral sequence of Hopkins and Miller, to calculate mapping spaces of commutative ring spectra. 

In this work, we construct the first examples of non-trivially $E_\infty$ topologically equivalent $E_\infty$ DGAs. One of these examples is in $E_\infty$ $\F_p$--DGAs. This is particularly interesting because one of the open questions in \cite{dugger2007topological} asks if there are any examples of non-trivial   topological equivalences of   $k$--DGAs for a field $k$. Our example provides a positive answer to this question in $E_\infty$ DGAs. Although there is an example of non-trivial $E_\infty$ topological equivalences over $\F_p$, our non-existence results for $E_\infty$ topological equivalences hint that such examples are not common. 

Before stating our non-existence results, we note that    topologically equivalent   DGAs have isomorphic homology rings. This is because the Quillen equivalence between   $R$--DGAs and   $HR$--algebras gives an isomorphism between the homology ring of an   $R$--DGA and the homotopy ring of the corresponding ring spectra. Therefore if $X$ and $Y$ are   topologically equivalent   DGAs, then $H_*(X) \cong \pi_*(HX) \cong \pi_*(HY) \cong H_*(Y)$ where the isomorphisms are ring isomorphisms and the isomorphism in the middle is induced by the   topological equivalence. The same is true for $E_\infty$ topological equivalences, but as Example \ref{example Fp} indicates, the isomorphism of homology rings may not preserve Dyer--Lashof operations. However, by Theorem \ref{nonexistence 3}, if $X$ and $Y$ are $E_\infty$ topologically equivalent $E_\infty$ $\F_p$--DGAs with trivial first homology, then the isomorphism of homology rings induced by the $E_\infty$ topological equivalence preserves Dyer--Lashof operations, ie it is an isomorphism of algebras over the Dyer--Lashof algebra.

For co-connective $E_\infty$ $\Z/(m)$--DGAs, we prove that there are no non-trivial $E_\infty$ topological equivalences where $m$ is a non-unital integer (ie $m \in \Z$ with $m\neq \pm 1$). This in particular implies that there are no non-trivial $E_\infty$ topological equivalences of co-connective $E_\infty$ $\Z$--DGAs.

\begin{theorem}
\label{nonexistence 2}
$E_\infty$ Topologically equivalent co-connective $E_{\infty}$  $\Z/(m)$--DGAs are quasi-isomorphic as $E_{\infty}$  $\Z/(m)$--DGAs. In other words, $E_\infty$ topological equivalences and quasi-isomorphisms agree for co-connective $E_\infty$ $\Z/(m)$--DGAs. 
\end{theorem}

There is an important class of examples for this theorem, namely the  cochain complex of a topological space with coefficients in $\Z/(m)$; this is the function spectrum $F(\Sigma^{\infty}X_+, H\Z/(m))$ for a topological space $X$. Note that since we use homological grading, the cochain complex of a space is co-connective. Moreover, Mandell's result states that finite type nilpotent spaces are weakly equivalent if and only if their cochain complexes with integer coefficients are quasi-isomorphic as $E_\infty$ $\Z$--DGAs \cite{mandell2006cochains}. Combining Mandell's result with Theorem \ref{nonexistence 2}, we obtain the following corollary.

\begin{corollary}
Finite type nilpotent spaces are weakly equivalent if and only if their cochain complexes with integer coefficients are $E_\infty$ topologically equivalent. 
\end{corollary}
\begin{remark}
It is also interesting to consider the following consequence of Theorem \ref{nonexistence 2}. For a co-connective commutative $\Sp$--algebra $X$, we use the $E_\infty$ connective cover $H\pi_0 X \to X$ to obtain a map $H\Z \to X$ which gives $X$ a commutative $H\Z$--algebra structure. This says that there is an $E_\infty$ $\Z$--DGA corresponding to a co-connective commutative $\Sp$--algebra. By this and Theorem \ref{nonexistence 2}, we deduce that weak equivalence classes of co-connective commutative $\Sp$--algebras are uniquely determined by the quasi-isomorphism classes of the corresponding $E_\infty$ $\Z$--DGAs.
\end{remark}

In Example \ref{example Fp}, we construct $E_\infty$ $\F_p$--DGAs that are non-trivially $E_\infty$ topologically equivalent. Therefore it is not possible to generalize Theorem \ref{nonexistence 2} to all $E_\infty$ $\F_p$--DGAs. However, for $E_\infty$ $\F_p$--DGAs with trivial first homology we have the following result.

\begin{theorem}
\label{nonexistence 3}
Let $X$ and $Y$ be $E_{\infty}$ $\F_p$--DGAs with trivial first homology group. If $X$ and $Y$ are $E_\infty$ topologically equivalent, then they are $H_{\infty}$ $\F_p$--algebra equivalent. Furthermore, an $\Sp$--algebra equivalence between $HX$ and $HY$ induces an isomorphism of the homology rings that preserves Dyer--Lashof operations.
\end{theorem}
We actually prove a stronger result. Theorem \ref{thm DL} states that for $H_\infty$ $H\F_p$--algebras with trivial first homotopy, $H_\infty$ $\Sp$--algebra equivalence implies $H_\infty$ $H\F_p$--algebra equivalence. 

The condition of trivial first homology is due to the fact that the dual Steenrod algebra is generated by an element of degree one as a ring with Dyer--Lashof operations. Again by Example \ref{example Fp}, this condition cannot be removed from this theorem. 

\begin{remark} \label{remark Tyler}
In \cite{lawson2015note}, Lawson produces examples of $H_\infty$ $\Sp$--algebras whose $H_\infty$ $\Sp$--algebra structures do not lift to commutative $\Sp$--algebra structures. One of the intermediate results of \cite{lawson2015note} states that Theorem \ref{thm DL} is still true without the restriction on the first homotopy but Example \ref{example Fp} contradicts this. The examples of spectra constructed in \cite{lawson2015note} are co-connective. Therefore, Theorem \ref{thm DL} recovers the main result of \cite{lawson2015note}. We elaborate on this in Section \ref{section dl}.
\end{remark}

The proof of the non-existence theorem in \cite{dugger2007topological} for  $\mathbb{Q}$--DGAs also works for $E_\infty$ $\mathbb{Q}$--DGAs. We obtain the following. 

\begin{theorem} \sloppy \label{theorem Q}
($E_\infty$) topologically equivalent ($E_\infty$) $\mathbb{Q}$--DGAs are quasi-isomorphic. That is, ($E_\infty$) topological equivalences and ($E_\infty$) 
quasi-isomorphisms agree in ($E_\infty$) $\mathbb{Q}$--DGAs.
\end{theorem}

In the next section, we explain the examples of non-trivial   topological equivalences given in \cite{dugger2007topological} and in the appendix, we make a correction to a mistake in the construction of these examples. Section \ref{Section obstruction} discusses the obstruction spectral sequences that we will use for calculating mapping spaces of ring spectra and Section \ref{Dyer Lashof} describes the dual Steenrod algebra and the Dyer--Lashof operations on it. Section \ref{section examples} is devoted to our examples of non-trivial $E_\infty$ topological equivalences. Section \ref{section nonexist} contains the proof of Theorem \ref{nonexistence 2} and Section \ref{section dl} contains the proof of Theorem \ref{nonexistence 3}.

\textbf{Notation} As noted earlier, for a commutative ring $R$, when we say $R$--DGAs we mean associative $R$--DGAs. Similarly for a commutative ring spectrum $R$, $R$--algebras denote associative $R$--algebras. A smash product without a subscript $\wedge$ denotes the smash product over the sphere spectrum. The category of spectra we use is symmetric spectra in topological spaces with the positive model structure as in Mandell, May, Schwede and Shipley \cite{Mandell01Model}. 

\textbf{Acknowledgements} The author would like to thank his thesis advisor Brooke Shipley for her guidance, advice and financial support over the years. I also would like to thank Michael Hopkins for suggesting Goerss Hopkins Miller obstruction theory for the purpose of studying topological equivalences and Paul Goerss for showing me the version of this obstruction theory that I use in this paper. Finally I would like to thank Benjamin Antieau for his helpful advice over the years I worked on this project.

\section{Previous examples of topological equivalences} \label{sec previous examples}

In this section we discuss the examples of non-trivial topological equivalences in \cite{dugger2007topological}. 

\begin{example} \label{ex ds}
There are exactly two non quasi-isomorphic DGAs whose homology is $\Lambda_{\F_p}(x_{2p-2})$, the exterior algebra over $\F_p$ with a single generator in degree $2p-2$, and these DGAs are topologically equivalent. Therefore they are non-trivially topologically equivalent. For $p=2$, one of these DGAs is the formal one and the other one is given by
\[ \Z[e_1;de_1 = 2]/(e_1^4) \ \textnormal{where} \ \lvert e_1  \rvert = 1. \]
\end{example}
This example is constructed by classifying weak equivalence classes of Postnikov extensions which are obtained using topological Hochschild cohomology. However, this construction in \cite{dugger2007topological} contains a gap. The weak equivalences classes calculated in \cite{dugger2007topological} are weak equivalence classes of Postnikov extensions. This in general may not correspond to weak equivalence classes that should be considered here, namely the weak equivalence classes of  $H\Z$--algebras. In the appendix, we explain this in detail and correct the mistake in \cite{dugger2007topological} by showing that these two equivalence classes agree for this particular example. 

\begin{example} \label{previous ex 2}
The second example of Dugger and Shipley has a simpler construction. They start with $H\Z \wedge H\F_2$ and give this $\Sp$--algebra two $H\Z$--algebra structures using the maps $H\Z \cong H\Z \wedge \Sp \to H\Z \wedge H\F_2$ and $H\Z \cong \Sp \wedge H\Z  \to H\Z \wedge H\F_2$. These two $H\Z$--algebras are not weakly equivalent but their underlying $\Sp$--algebras are the same. This means that we have two DGAs that are not quasi-isomorphic but are topologically equivalent. In Theorem \ref{examples 3}, we provide a generalization of this example in $E_\infty$ DGAs. 
\end{example}

\section{Obstruction theories for ring spectra} \label{Section obstruction}

For a commutative $\Sp$--algebra $X$, a commutative $H\Z$--algebra structure on $X$ is given by a map $H\Z \to X$ of commutative $\Sp$--algebras. In other words, the category of commutative $H\Z$--algebras is the category commutative $\Sp$--algebras under $H\Z$. Therefore it is natural to consider the maps from $H\Z$ to a commutative $\Sp$--algebra $X$ for the purpose of studying $E_\infty$ topological equivalences. For this, we employ an obstruction spectral sequence to calculate homotopy class of maps in commutative ring spectra.

The obstruction spectral sequence we use relies on Bousfield's obstruction spectral sequence \cite{bousfield1989homotopy}. The first application of Bousfield's obstruction theory to ring spectra was in the Hopkins--Miller theorem, see \cite{rezk1998notes}. The obstruction theory of Hopkins and Miller is for associative ring spectra. It is used for showing existence of ring structures on spectra and for calculating mapping spaces of ring spectra. Hopkins and Miller use this obstruction theory to show that the Morava stabilizer group acts on the Morava $E$--theory spectrum $E_n$. Later, Goerss and Hopkins generalized this theory to commutative ring spectra \cite{goerss2004moduli}.

Johnson and Noel generalized the obstruction theory of the Hopkins Miller theorem to calculate mapping spaces of algebras over a general monad in a model category in \cite{johnson2014lifting}. This is called the $T$--algebra spectral sequence. 

Generalizing the obstruction theory of the Hopkins--Miller theorem to commutative ring spectra is not trivial because of the following problem. For a spectrum $X$ and a homology theory $E_*$ corresponding to another spectrum $E$, if $E_*X$ is flat over $E_*$, then $E_* T(X)$ is the free associative $E_*$--algebra over $E_*X$ where $T(X)$ is the free associative ring spectrum over $X$. For commutative ring spectra, one uses the free commutative ring spectra functor $\Pp_{\Sp}$ but $E_*\Pp_{\Sp}(X)$ may not have a nice description, even under the above flatness assumption. However, for calculating mapping spaces of commutative $H\F_p$--algebras one uses the fact that $\Pp_{H\F_p}(X)_*$ is the free unstable algebra over the Dyer--Lashof algebra generated by $X_*$. Noel uses this with the results of \cite{johnson2014lifting} and constructs a spectral sequence that calculates mapping spaces of commutative $H\F_p$--algebras, see Proposition 2.2 of \cite{noel2015t}. More generally, his spectral sequence calculates mapping spaces of commutative $Hk$--algebras for any field $k$.  Using the adjunction between commutative $\Sp$--algebras and commutative $H\F_p$--algebras we obtain the following spectral sequence from Noel's spectral sequence. 

\begin{theorem} \label{obstruction} 
\cite[Proposition 2.2]{noel2015t}
     Let $X$ be a commutative $\Sp$--algebra and let $Y$ be a commutative $H\F_p$--algebra. Given a map $\phi \co X \to Y$ of commutative $\Sp$--algebras, there is a spectral sequence abutting to $\pi_{t-s} \mathrm{map}_{\Sp \mhyphen  cAlg}(X,Y)$ where $\Sp \mhyphen  cAlg$ denotes commutative $\Sp$--algebras. The $E_2$ term of this spectral sequence is given by
    \[E_2^{0,0} = \mathrm{Hom}_{\mathcal{R} \mhyphen alg} ({H\F_p}_*X,Y_*) \] and for $t>0$,
    \[E_2^{s,t} = \mathrm{Der}^s_{\mathcal{R} \mhyphen alg} ({H\F_p}_*X,Y^{S^t}_*) \] 
    where $\mathrm{Der}^s_{\mathcal{R}\mhyphen alg}(-,-)$ denotes the $s$th Andr{\'e}--Quillen cohomology for unstable algebras with Dyer--Lashof operations \cite{quillen1970co}, $Y^{S^t}$ denotes the mapping spectrum from the $t$--sphere to $Y$ and $\mathrm{Hom}_{\mathcal{R}\mhyphen alg} ({H\F_p}_*X,Y_*)$ denotes morphisms preserving Dyer--Lashof operations.
    \\
    \\
    Obstructions to lifting a morphism in $E_2^{0,0}$ to a morphism of commutative $\Sp$--algebras lie in $\mathrm{Der}^t_{\mathcal{R}\mhyphen alg} ({H\F_p}_*X,Y^{S^{t-1}}_*)$ for $t \geq 2$. 
    \\
    \\
    Obstructions to up-to homotopy uniqueness of a lift lie in \\ 
    $\mathrm{Der}^t_{\mathcal{R}\mhyphen alg} ({H\F_p}_*X,Y^{S^t}_*)$ for $t \geq 1$.

\end{theorem}

\begin{proof}
The adjunction between commutative $\Sp$--algebras and commutative $H\F_p$--algebras gives
\[\mathrm{map}_{\Sp \mhyphen cAlg}(X,Y) \cong \mathrm{map}_{H\F_p \mhyphen cAlg}(H\F_p \wedge X,Y).\]
Therefore the setting of Noel's spectral sequence that calculates the homotopy groups of $\mathrm{map}_{H\F_p \mhyphen cAlg}(H\F_p \wedge X,Y)$ provides us the spectral sequence above.

Noel's spectral sequence is a special case of the $T$--algebra spectral sequence of \cite{johnson2014lifting}. Therefore Theorem 4.5 of \cite{johnson2014lifting} gives us the obstruction theoretical results.
\end{proof}

\section{Dyer--Lashof operations and the dual \\ 
Steenrod algebra} \label{Dyer Lashof}

 For a commutative ring spectrum $R$ we denote the free commutative algebra functor from $R$--modules to commutative $R$--algebras by $\Pp_R$. This functor is homotopically well behaved and induces a monad on $Ho(R \mhyphen mod)$ and the algebras over this monad are called $H_\infty$ $R$--algebras. Therefore, an $E_\infty$ algebra is an $H_\infty$ algebra. The converse to this is shown to be false by counter-examples in \cite{noel2009h} and \cite{lawson2015note}.  

Dyer--Lashof operations are power operations, just like the Steenrod operations, that are constructed in a way to act on the homotopy ring of $H_\infty$ $H\F_p$--algebras in \cite{brunerh}. Equivalently, they act on the homology ring of $H_\infty$ $\F_p$--DGAs. Indeed the category of $H_\infty$ $H\F_p$--algebras is equivalent to the category of graded commutative rings over $\F_p$ with Dyer--Lashof operations satisfying the allowability and $p$th power conditions, which are called unstable algebras over the Dyer--Lashof algebra, see the discussion in section 3 of \cite{lawson2015note}. 

For each integer $s$, there is a Dyer--Lashof operation denoted by $\mathrm{Q}^s$. These operations are preserved under $H_\infty$ $H\F_p$--algebra morphisms and hence $E_\infty$ $H\F_p$--algebra morphisms. 
The operation $\mathrm{Q}^s$ increases the degree by $2s(p-1)$ for odd primes and by $s$ for $p=2$. For an element $x$ in the homotopy ring of a commutative $H\F_p$--algebra, the unstable Dyer--Lashof operations satisfy the following properties (the properties for $p=2$ are given in parentheses). 
\begin{equation*}
    \mathrm{Q}^sx = 0 \ \text{for} \  2s < \lvert x \rvert \  (\text{for} \  s < \lvert x  \rvert )
\end{equation*}
\begin{equation*}
    \mathrm{Q}^sx = x^p \ \text{for} \  2s = \lvert x \rvert \  (\text{for} \  s = \lvert x  \rvert )
\end{equation*}
\begin{equation*}
    \mathrm{Q}^s 1 = 0 \ \text{for} \ s \neq 0
\end{equation*}
Also, these operations satisfy the Cartan formula and the Adem relations as in Chapter III Theorem 1.1 of \cite{brunerh}.

As mentioned earlier, for a commutative $H\F_p$--algebra $X$, $\Pp_{H\F_p}(X)_*$  is the free unstable algebra over the Dyer--Lashof algebra 
 generated by $X_*$.   
\begin{theorem} \label{BakerfreeDL} \cite{baker2012calculating}
     $\Pp_{H\F_p}(X)_*$ is the free commutative graded $\F_p$--algebra generated by $\mathrm{Q}^I x_j$ where $x_j$'s form a basis for $X_*$ and $I = (\varepsilon_1,i_1,\varepsilon_2,...,\varepsilon_n,i_n)$ is admissible and satisfies  \textnormal{excess}$(I) + \varepsilon_1 >\lvert x_j \rvert$. 
\end{theorem} 

The definition of admissibility and excess can be found in \cite{may1971homology}.
\\

\textbf{Dual Steenrod Algebra.} Now we discuss the dual Steenrod algebra and the Dyer--Lashof operations on it. The dual Steenrod algebra is first described by Milnor in \cite{milnor1958steenrod} and the Dyer--Lashof operations on it are first studied in Chapter III of \cite{brunerh}. We also recommend \cite{baker2013power}. The dual Steenrod algebra $\mathcal{A}_* \cong {H\F_p}_* H\F_p $ is a free graded commutative $\F_p$--algebra. For $p=2$, it is given by two different standard sets of generators 
\[ \mathcal{A}_* = \F_2[\xi_r \ \lvert \ 
 r\geq 1] = \F_2[\zeta_r \ \lvert \ r \geq 1] \]

\noindent where $\lvert \xi_r \rvert = \lvert \zeta_r \rvert =  2^r-1$. The transpose map of the smash product applied to $H\F_p \wedge H\F_p$ induces an automorphism of the dual Steenrod algebra denoted by $\chi$. The reason we have two different set of generators above is to keep track of the action of $\chi$. We have $\chi(\xi_r)=\zeta_r$. The generating series $\xi(t) = t +  \Sigma_{r \leq 1} \xi_r t^{2r}$ and $\zeta(t) = t + \Sigma_{r \leq 1} \zeta_r t^{2r} $ are composition inverses in the sense that $\zeta(\xi(t)) = t = \xi (\zeta(t))$. This in particular shows that $\xi_1 = \zeta_1$.

Since commutative $H\F_p$--algebras form the category of commutative $\Sp$--algebras under $H\F_p$,  $H\F_p \wedge H\F_p$ can be given two different commutative $H\F_p$--algebra structures using the maps $H\F_p = H\F_p \wedge \Sp \to H\F_p \wedge H\F_p$ and $H\F_p = \Sp \wedge H\F_p \to H\F_p \wedge H\F_p$. We call these maps $g_1$ and $g_2$ respectively. We denote the Dyer--Lashof operations induced on $\mathcal{A}_*$ from the first structure map by $\mathrm{Q}^s$ and the second structure map by $\widetilde{\mathrm{Q}}^s$. 

Since the transpose map induces an isomorphism of commutative $H\F_p$--algebras, $\chi$ preserves the corresponding Dyer--Lashof operations ie $\chi(\mathrm{Q}^s x) = \widetilde{\mathrm{Q}}^s \chi(x)$. For $p=2$, $\mathcal{A}_*$ is generated as an algebra over the Dyer--Lashof algebra by $\xi_1$, and we have 
\[ \mathrm{Q}^{2^s-2} \xi_1 = \zeta_s \ \ \text{for} \ s \geq 1. \]
By using the fact that $\chi$ preserves Dyer--Lashof operations, one obtains 
\[ \widetilde{\mathrm{Q}}^{2^s-2} \xi_1 = \xi_s \ \ \text{for} \ s \geq 1. \]

For an odd prime $p$, the dual Steenrod algebra is given by the following. 
\[ \mathcal{A}_* = \F_p[\xi_r \ \lvert \ 
 r\geq 1] \otimes \Lambda(\tau_s \ \lvert \  s \geq 0) = \F_p[\zeta_r \ \lvert \ r \geq 1] \otimes \Lambda(\overline{\tau}_s \ \lvert \  s \geq 0) \]
 
\noindent Where $\lvert \xi_r \rvert = \lvert \zeta_r \rvert =  2(p^r-1)$ and $\lvert \tau_s \rvert = \lvert \overline{\tau}_s \rvert =  2p^s-1$. The action of the antipode map is given by $ \chi(\xi_r) =\zeta_r$ and $\chi(\tau_r)=\overline{\tau}_r $. We use the following formula to relate the two set of generators for the dual Steenrod Algebra, see Section 7 of \cite{milnor1958steenrod} and the proof of Lemma 4.7 in \cite{baker2013power}.
\begin{equation} \label{generators}
\bar{\tau}_s+\bar{\tau}_{s-1}\xi^{p^{s-1}}_1 + \bar{\tau}_{s-2}\xi^{p^{s-2}}_2 + \cdots+ \bar{\tau}_0 \xi_s + \tau_s = 0
\end{equation}

The dual Steenrod algebra is generated by $\tau_0$ as an algebra over the Dyer--Lashof algebra. We have the following formulae for the Dyer--Lashof operations for $s \geq 1$
\[ \mathrm{Q}^{(p^s-1)/(p-1)} \tau_0 = (-1)^s \overline{\tau}_s \]
\[\beta \mathrm{Q}^{(p^s-1)/(p-1)} \tau_0 = (-1)^s \zeta_s. \]
More can be found on the Dyer--Lashof operations on the dual Steenrod algebra in \cite{baker2013power} and in Chapter III of \cite{brunerh}.


\section{Examples of non-trivial $E_\infty$ topological equivalences} \label{section examples}
\subsection{Examples in $E_\infty$ $\F_p$--DGAs}

Here, we discuss the first examples of $E_\infty$ DGAs that are not quasi-isomorphic but are $E_\infty$ topologically equivalent, ie non-trivially $E_\infty$ topologically equivalent. 

The first example we construct is in $E_\infty$ $\F_p$--DGAs.
By the equivalence of $E_\infty$ $\F_p$--DGAs and commutative $H\F_p$--algebras, constructing non-trivially topologically equivalent $E_\infty$ $\F_p$--DGAs is the same as constructing commutative $H\F_p$--algebras that are not weakly equivalent as commutative $H\F_p$--algebras but are weakly equivalent as commutative $\Sp$--algebras. 

As we noted earlier, commutative $H\F_p$--algebras form the category of commutative $\Sp$--algebras under $H\F_p$. There is a model structure induced on the under-category where the weak equivalences, cofibrations and fibrations are precisely the same as for commutative $\Sp$--algebras. 

In our example, we start with a commutative $\Sp$--algebra $X$ and induce two different commutative $H\F_p$--algebra structures on this object by providing two different  commutative $\Sp$--algebra maps from $H\F_p$ to $X$. Clearly these two commutative $H\F_p$--algebras are weakly equivalent (even isomorphic) as commutative $\Sp$--algebras.
We show that these two commutative $H\F_p$--algebras are not weakly equivalent as commutative $H\F_p$--algebras by showing that their homotopy rings are not isomorphic as algebras over the Dyer--Lashof algebra. By the discussion of Section \ref{Dyer Lashof} this shows that these non-trivially $E_\infty$ topologically equivalent $E_\infty$ $\F_p$--DGAs are furthermore not equivalent as $H_\infty$ $\F_p$--DGAs. 
\begin{example} \label{example Fp}
For an odd prime $p$, the $E_\infty$ $\F_p$--DGAs we produce have the same homology ring given by $ \Lambda_{\F_p}[\tau_0,\xi_1,\tau_1] / (\tau_0 \tau_1,\tau_1\xi_1,\tau_0 \xi_1 - \tau_1)$ where the degrees of the generators are those of the dual Steenrod algebra ie  $\lvert \tau_0 \rvert = 1$, $\lvert \xi_1 \rvert = 2(p-1)$ and $\lvert \tau_1 \rvert = 2p-1$. However, the homology groups of these $E_\infty$ $\F_p$--DGAs are not isomorphic as algebras over the Dyer--Lashof algebra. In one of them, $\mathrm{Q}^1(\tau_0) = \tau_1$ and in the other, $\mathrm{Q}^1(\tau_0) = 0$. Therefore these two $E_\infty$ $\F_p$--DGAs are not equivalent as $H_\infty$ $\F_p$--DGAs and therefore they are not quasi-isomorphic.

For $p=2$ the homology ring of the two $E_\infty$ topologically equivalent $E_\infty$ $\F_2$--DGAs are $\F_2[\xi_1]/(\xi_1^4)$ where $\lvert \xi_1 \rvert= 1$ as in the dual Steenrod algebra. In the homology of the first $E_\infty$ $\F_2$--DGA, $\mathrm{Q}^2(\xi_1)= \xi_1^3$ and in the other one, $\mathrm{Q}^2(\xi_1) = 0$. Again, these two $E_\infty$ $\F_2$--DGAs are not quasi-isomorphic because their homology rings are not isomorphic as algebras over the Dyer--Lashof algebra and therefore they are not equivalent as $H_\infty$ $\F_2$--DGAs.

First, we discuss our example for $p=2$. For this, we are going to use Postnikov sections for commutative ring spectra. These were first introduced in Section 8 of \cite{basterra1999andre} for studying Postnikov towers of commutative ring spectra. The $n$th Postnikov section of a connected commutative ring spectrum $Z$ is a map $Z\to P_n Z$ that induces an isomorphism on $\pi_i (Z) \to \pi_i(P_n Z)$ for $i\leq n$ and for which $\pi_i P_nZ=0$ for $i>n$. Let $H\F_2 \wedge H\F_2 \to P_3 (H\F_2 \wedge H\F_2)$ be the third Postnikov section of $H\F_2 \wedge H\F_2$ as a commutative $\Sp$--algebra, we have $\pi_* (P_3 (H\F_2 \wedge H\F_2)) = \F_2 [\xi_1,\xi_2] / (\xi_1^4, \xi_2^2, \xi_1 \xi_2)$ with $\lvert \xi_1 \rvert = 1$ and $\lvert \xi_2 \rvert= 3$ ie the dual Steenrod algebra quotiented out by the ideal of elements of degree 4 and higher. By using Lemma \ref{cell} we kill the element $\xi_1^3 + \xi_2$ in $\pi_*(P_3 (H\F_2 \wedge H\F_2))$ and then taking the third Postnikov section, we obtain another commutative $\Sp$--algebra $X$ with $\pi_* (X) = \F_2 [\xi_1,\xi_2] / (\xi_1^4, \xi_2^2, \xi_1 \xi_2,\xi_1^3 + \xi_2) = \F_2[\xi_1]/(\xi_1^4)$. The reason we kill $\xi_1^3+ \xi_2$ is because it is equal to $\zeta_2$ in the dual Steenrod algebra, this follows from the generating series we discuss in Section \ref{Dyer Lashof}. Note that for Lemma \ref{cell}, one can use the commutative $H\F_2$--algebra structure on $P_3(H\F_2 \wedge H\F_2)$ induced by the map $H\F_2 \wedge \Sp \to H\F_2 \wedge H\F_2 \to P_3(H\F_2 \wedge H\F_2)$.

Furthermore, we have a map $P_3 (H\F_2 \wedge H\F_2) \to X$ with the induced map on the homotopy rings being the canonical one. By pre-composing this map with the map into the Postnikov section $H\F_2 \wedge H\F_2 \to P_3 (H\F_2 \wedge H\F_2)$, we obtain a map of commutative $\Sp$--algebras $f\co H\F_2\wedge H\F_2 \to X$. We construct two commutative $\Sp$--algebra maps from $H\F_2$ to $X$ as shown in the diagram below. 
 \begin{equation} \label{diag Fp example}
 \begin{tikzcd}
 H\F_2 \cong H\F_2 \wedge \Sp \arrow[rd,"g_1"]
 \arrow[dr]
 &
 &
 \\
 &H\F_2 \wedge H\F_2 \arrow[r,"f"]    &X 
 \\
 H\F_2 \cong \Sp \wedge H\F_2  
 \arrow[ur,swap,"g_2"]
 \end{tikzcd}
 \end{equation}
The maps $g_1$ and $g_2$ induce two commutative $H\F_2$--algebra structures on $X$. The 
commutative $H\F_2$--algebra with unit $f \circ g_1$ is denoted by  $X_1$ and with unit $f\circ g_2$ by $X_2$. 

As we discuss in Section \ref{Dyer Lashof}, $H\F_2 \wedge H\F_2$ can be given two commutative $H\F_2$--algebra structures through the maps $g_1$ and $g_2$ and we call the two associated commutative $H\F_2$--algebras $Y_1$ and $Y_2$. The Dyer--Lashof operations on $\pi_*(Y_1)$ are denoted by $\mathrm{Q}^s$ and the Dyer--Lashof operations on $\pi_*(Y_2)$ are denoted by $\widetilde{\mathrm{Q}}^s$ as in Section \ref{Dyer Lashof}. 

Because morphisms in commutative $H\F_2$--algebras are morphisms of commutative $\Sp$--algebras under $H\F_2$, from the map $f$ alone we obtain two $H\F_2$--algebra maps $g\co Y_1 \to X_1$ and $h \co Y_2 \to X_2$. These maps induce maps that preserve Dyer--Lashof operations in the homotopy rings and we use this to understand the Dyer--Lashof operations on $\pi_*(X_1)$ and $\pi_*(X_2)$. On $\pi_*(X_1)$, 
\[\mathrm{Q}^2(\xi_1) = \mathrm{Q}^2(g_*(\xi_1)) = g_*(\mathrm{Q}^2(\xi_1)) = g_*(\zeta_2)= \zeta_2= \xi_1^3 +\xi_2 = 0.\] 
On $\pi_*(X_2)$, 
\[\mathrm{Q}^2(\zeta_1) =\mathrm{Q}^2(h_*(\zeta_1))=  h_*(\widetilde{\mathrm{Q}}^2(\zeta_1)) = h_*(\xi_2)= \xi_2 \neq 0.\] 

Therefore, $\pi_*(X_1)$ and $\pi_*(X_2)$ are not isomorphic as algebras over the  Dyer--Lashof algebra as desired.



For odd primes $p$, the construction of an example is similar. By Equation \eqref{generators} in Section \ref{Dyer Lashof}, $\bar{\tau}_1 = \tau_0 \xi_1- \tau_1$. Therefore one can use Lemma \ref{cell} to kill $\tau_0 \xi  - \tau_1$ which kills $\mathrm{Q}^1 \tau_0$. The rest of the arguments follow similarly.

\end{example}

\begin{lemma} \label{cell}

Let $X$ be a connective commutative $H\F_p$--algebra with\\ $\pi_0(X)= \F_p$  and $\pi_i(X)= 0$ for $i>n$. Given $x \in \pi_n(X) $ there is a commutative $\Sp$--algebra $Y$ and a map of commutative $\Sp$--algebras $X \to Y$ which induces the morphism $X_* \to X_*/(x)$ on the level of homotopy groups.
\end{lemma}

\begin{proof} 
Let the $H\F_p$--module map $\Sigma^n H\F_p \to X$ represent $x$. By adjunction and by applying the $n$th Postnikov section functor, we obtain the map $P_n( \Pp_{H\F_p} (\Sigma^n H\F_p)) \to P_n(X)\simeq X$. The homotopy ring $\pi_*(\Pp_{H\F_p} (\Sigma^n H\F_p))$ is the free unstable algebra over the  Dyer--Lashof algebra generated by an element of degree $n$. Therefore we obtain that $\pi_*(P_n( \Pp_{H\F_p} (\Sigma^n H\F_p))) = \Lambda_{\F_p} [x_n]$ where $\lvert x_n \rvert = n$. Let $Z$ denote $P_n( \Pp_{H\F_p} (\Sigma^n H\F_p))$. The required $Y$ is $P_n(H\F_p \wedge_Z X)$ where $H\F_p$ is a commutative $Z$--algebra by the map $Z \to P_0Z = H\F_p$. 

The homotopy groups of $H\F_p \wedge_Z X$ can be calculated using the K{\"u}nneth spectral sequence whose $E^2$ page is $\text{Tor}^{\Lambda_{\F_p}[x_n]}_{*,*}(\F_p,X_*)$. We want to show that for degree less than $n+1$, $\pi_*(H\F_p \wedge_Z X)$  and $ \F_p \otimes_{\Lambda_{\F_p}[x_n]} X_*$ agree.

There is a standard resolution of $\F_p$ over $\Lambda_{\F_p}[x_n]$ which is $\Sigma^{kn} \Lambda_{\F_p} [x_n]$ at homological degree $k$.  Therefore we have $\text{Tor}^{\Lambda_{\F_p}[x_n]}_{k,l}(\F_p,X_*)= 0$ for $k>0$ and $l< n$ and hence the only terms that contribute to $\pi_i(H\F_p \wedge_Z X)$ for $i \leq n$ are in $E^2_{0,*}$. Since the differentials on these terms hit the second quadrant, they are zero. We obtain $\pi_i (H\F_p \wedge_Z X) \cong \text{Tor}^{\Lambda_{\F_p}[x_n]}_{0,i}(\F_p,X_*) \cong (\F_p \otimes_{\Lambda_{\F_p}} X_*)_i $ for $i \leq n$. Because $\F_p \otimes_{\Lambda_{\F_p}[x_n]} X_*$ is concentrated in degrees between $0$ and $n$, $\pi_*Y \cong \pi_* (P_n( H\F_p \wedge_Z X)) \cong \F_p \otimes_{\Lambda_{\F_p}} X_*$.  Since the image of $x_n$ in $\pi_*(X)$ is $x$, we have $\pi_*Y \cong \F_p \otimes_{\Lambda_{\F_p}} X_* \cong \pi_*(X)/(x)$ and $X \to Y$ induces the desired morphism $\pi_*X \to \pi_*X/(x)$ on the homotopy ring. For this Lemma, we forget the $H\F_p$ structure and use the corresponding commutative $\Sp$--algebra map $X \to Y$.

\end{proof}
\subsection{Examples in $E_\infty$ $\Z$--DGAs}

Now we discuss examples of non-trivial $E_\infty$ topological equivalences in $E_\infty$ $\Z$--DGAs. Theorem \ref{examples 2} below gives a general scenario where examples of non-trivially $E_\infty$ topologically equivalent $E_\infty$ $\Z$--DGAs occur.  
We start with Theorem \ref{examples 3} which provides examples that have a simple construction. This theorem uses the construction of the example of Dugger and Shipley that we discuss in Example \ref{previous ex 2} above. For $E_\infty$ DGAs, this construction is generalized to odd primes. 

    \begin{theorem} \label{examples 3} 
    Let $X$ and $Y$ denote the commutative $H\Z$--algebras whose underlying commutative $\Sp$--algebras are $H\Z \wedge H\F_p$ and whose commutative $H\Z$--algebra structures are given by the map 
    \[H\Z \cong H\Z \wedge \Sp \to H\Z \wedge H\F_p \]
    and the map 
    \[H\Z \cong \Sp \wedge H\Z \to H\Z \wedge H\F_p \]
    respectively. 
    
    These commutative $H\Z$--algebras $X$ and $Y$ are not weakly equivalent. Since the underlying commutative $\Sp$--algebras of $X$ and $Y$ are the same, we deduce that the $E_\infty$ $\Z$--DGAs corresponding to $X$ and $Y$ are not quasi-isomorphic but are $E_\infty$ topologically equivalent.
    
    \end{theorem}
    \begin{proof}
    Assume that $X$ and $Y$ are weakly equivalent as commutative $H\Z$--algebras. Taking cofibrant and fibrant replacements, we assume that there is a weak equivalence $\begin{tikzcd}[sep=small] \psi  \co X \arrow[r,"\sim"]&Y\end{tikzcd}$ of commutative $H\Z$--algebras. This means that there is the following diagram in commutative $\Sp$--algebras. 
         \begin{equation} \label{diag ex 3}
 \begin{tikzcd}
& H\Z \arrow[rd,"\varphi_Y"] \arrow[ld,swap,"\varphi_X"]
\\
 U(X) \arrow[rr,"\simeq",swap] \arrow[rr,"\psi"]
 & & U(Y)
 \end{tikzcd}
 \end{equation}
 Here, $U$ denotes the forgetful functor to commutative $\Sp$--algebras and $\varphi_X$ and $\varphi_Y$ denote the commutative $H\Z$--algebra structure maps of $X$ and $Y$ respectively.     Note that by the Künneth spectral sequence,
    \[{H\F_p}_*U(X) \cong {H\F_p}_* (H\Z \wedge H\F_p) \cong {H\F_p}_*H\Z \otimes_{\F_p} {H\F_p}_* H\F_p.\]
 Taking the $H\F_p$ homology of the diagram above, we obtain the following.
    \begin{equation} \label{diag 2 ex 3}
 \begin{tikzcd}
& {H\F_p}_* H\Z \arrow[rd,"{H\F_p}_* \varphi_Y"] \arrow[ld,swap,"{H\F_p}_* \varphi_X"]
\\
 {H\F_p}_*H\Z \otimes_{\F_p} {H\F_p}_* H\F_p \arrow[rr,"\simeq",swap] \arrow[rr,"{H\F_p}_* \psi"]
 & & {H\F_p}_*H\Z \otimes_{\F_p} {H\F_p}_* H\F_p
 \end{tikzcd}
 \end{equation}

    With this identification, ${H\F_p}_* \varphi_X (x) = x \otimes 1$. Similarly, we have ${H\F_p}_*U(Y) \cong {H\F_p}_*H\Z \otimes_{\F_p} {H\F_p}_* H\F_p$ and ${H\F_p}_* \varphi_Y (x) = 1 \otimes x$. As noted earlier, the canonical map ${H\F_p}_* H\Z \to {H\F_p}_* H\F_p$ is an inclusion and ${H\F_p}_* H\Z$ is a free commutative ring generated by the same generators as ${H\F_p}_* H\F_p$ except $\tau_0$. The only degree one element in ${H\F_p}_*H\Z\otimes_{\F_p} {H\F_p}_* H\F_p$ is $1 \otimes \tau_0$ and since ${H\F_p}_* \psi$ is an isomorphism, this is mapped by ${H\F_p}_* \psi$ to $1 \otimes \tau_0$. Since $\mathrm{Q}^1 (1 \otimes \tau_0) = 1 \otimes \overline{\tau}_1$, ${H\F_p}_* \psi (1 \otimes \overline{\tau}_1) = 1\otimes \overline{\tau}_1$. But the commutativity of the triangle forces ${H\F_p}_* \psi (\overline{\tau}_1 \otimes 1) = 1 \otimes \overline{\tau}_1$ and this contradicts the injectivity of ${H\F_p}_* \psi$. Therefore $X$ and $Y$ are not equivalent as commutative $H\Z$--algebras. The argument for $p=2$ is similar.

    \end{proof}
    
\begin{theorem} \label{examples 2}
    Let $X$ be either an $E_\infty$ $\F_p$--DGA for an odd prime $p$ that satisfies

    \begin{enumerate}
        \item $H_i X = 0$ for $i= 1$ and $i>2p^2-4$
        \item $H_i X \neq 0$ for either $i=2p-1$ or $i=2p-2$
    \end{enumerate}
    or an $E_\infty$ $\F_2$--DGA that satisfies
       \begin{enumerate}
        \item $H_i X = 0$ for $i=1$ and $i>4$
        \item $H_i X \neq 0$ for $i=2$.
    \end{enumerate}
   For such an $X$, there exists an $E_\infty$ $\Z$--DGA that is $E_\infty$ topologically equivalent to $X$ but not quasi-isomorphic to $X$.
   \end{theorem}
   
   Indeed, we show that the $E_\infty$ $\Z$--DGA we construct in the proof of the above theorem is not quasi-isomorphic to any $E_\infty$ $\F_p$--DGA.
   
   It is clear that $E_\infty$ $\F_p$--DGAs that satisfy the conditions of this theorem exist. One can start with a graded commutative ring that satisfies the above conditions and use the corresponding formal commutative $\F_p$--DGA.
   
   \begin{remark}
   
   A special case of the above theorem gives the example of Dugger and Shipley that we discuss in Example \ref{ex ds}. For this, one uses the formal commutative $\F_2$--DGA with homology the exterior algebra with a generator in degree 2 for $X$. Moreover, our theorem provides a generalization of this example. We start with an $X$ as described in the $p=2$ case above and let $Y$ be the $E_\infty$ $\Z$--DGA we produce in the proof, then $X$ and $Y$ are not quasi-isomorphic as associative $\Z$--DGAs although they are topologically equivalent ie $X$ and $Y$ are non-trivially topologically equivalent as associative DGAs.
   \end{remark}
   \begin{proof}
   
   To produce our $E_\infty$ $\Z$--DGA, we start with an $X$ as above, and construct a commutative $\Sp$--algebra map $\varphi_Y \co H\Z \to U(X)$ using obstruction theory where $U$ is the forgetful functor to commutative $\Sp$--algebras. This gives us a new commutative $H\Z$--algebra $Y$ whose underlying commuative $\Sp$--algebra is $U(X)$. Obstruction theory gives us control over the map induced by $\varphi_Y$ on the $H\F_p$ homology. Using this and Dyer--Lashof operations, we show that $Y$ is not  weakly equivalent as a commutative $H\Z$--algebra to any commutative $H\F_p$--algebra.
   
   We describe our example when $p$ is an odd prime and 
  $H_{2p-1}X \neq 0$. The case $H_{2p-2} X \neq 0$ is similar. We explain the $p=2$ example at the end.
  
   For the obstruction spectral sequence of Theorem \ref{obstruction}, we use the composite map $H\Z \to H\F_p \xrightarrow{\varphi_X} U(X)$ as a base-point where the map $\varphi_X$ is the $H\F_p$ structure morphism of $X$. Using this base point and by setting up the obstruction spectral sequence to calculate the commutative $\Sp$--algebra maps from $H\Z$ to $X$, we obtain that obstructions to lifting a morphism of unstable algebras over the Dyer--Lashof algebra in 
    \[ E_2^{0,0} = \mathrm{Hom}_{\mathcal{R}\mhyphen alg} ({H\F_p}_*H\Z,X_*) \] 
    to a commutative $\Sp$--algebra map from $H\Z$ to $U(X)$ lie in the cohomology groups 
    \[\mathrm{Der}^t_{\mathcal{R}\mhyphen alg} ({H\F_p}_*H\Z,X^{S^{t-1}}_*)\  \text{for} \ t \geq 2.\]
    Here, $X^{S^{t-1}}$ denotes the mapping spectrum from the $t-1$-sphere $S^{t-1}$ to $X$. In Lemma \ref{obstructions zero} below we show that these groups that contain the obstructions are trivial. Therefore every map in 
    $\mathrm{Hom}_{\mathcal{R}\mhyphen alg} ({H\F_p}_*H\Z,X_*)$ lifts to a commutative $\Sp$--algebra map. 
    
    The canonical morphism ${H\F_p}_* H\Z \to {H\F_p}_*H\F_p$ is an injection and the image of this morphism is the free commutative algebra generated by $\zeta_i$ and $\overline{\tau}_i$ for $i \geq 1$ when $p$ is odd. For $p=2$, the image is generated by $\xi_1^2$ and $\zeta_i$ for $i \geq 1$. Since the above inclusion comes from a map of commutative $H\F_p$--algebras, it preserves Dyer--Lashof operations. This says that the Dyer--Lashof operations on ${H\F_p}_* H\Z$ are those of ${H\F_p}_* H\F_p$. 
    
    To construct the commutative $\Sp$--algebra map $\varphi_Y \co H\Z \to U(X)$ that defines $Y$ as desired, we start with any morphism $f$ in $E_2^{0,0} = \mathrm{Hom}_{\mathcal{R}\mhyphen alg} ({H\F_p}_*H\Z,X_*)$ that maps $\tau_1$ to a nonzero element in $H_{2p-1}X$ and use the lift of this map to a commutative $\Sp$--algebra map. Here, $\varphi_Y$ being a lift of $f$ means that the map 
    \[H\F_p \wedge H\Z \xrightarrow{\text{id}\wedge \varphi_Y} H\F_p \wedge U(X) \to U(X)\]
    induces $f$ in homotopy where the second map is given by the $H\F_p$--module structure map of $X$.
    
    At this point, we need to show that the $E_\infty$ $\Z$--DGAs corresponding to $X$ and $Y$ are not quasi--isomorphic; ie, $X$ and $Y$ are not weakly equivalent as commutative $H\Z$--algebras. 
    Assume that they are weakly equivalent over $H\Z$. We start with an $X$ that is cofibrant and fibrant as a commutative $H\F_p$--algebra. Therefore, $Y$ is also fibrant as a commutative $H\Z$--algebra because the underlying commutative $\Sp$--algebra of $Y$ is the underlying commutative $\Sp$--algebra of $X$. Our $X$ is also  cofibrant as a commutative $H\Z$--algebra since we use a cofibrant $H\F_p$ so that the initial map$\begin{tikzcd}[sep=small]H\Z \arrow[r,hookrightarrow]&H\F_p \arrow[r,hookrightarrow]&X \end{tikzcd}$is a composition of two cofibrations. Recall that cofibrations of commutative $H\F_p$--algebras are those of commutative $H\Z$--algebras.  Therefore, we have a weak equivalence of commutative $H\Z$--algebras$\begin{tikzcd}[sep=small] \psi  \co X \arrow[r,"\sim"]&Y.\end{tikzcd}$ That is, we have the following commuting diagram in commutative $\Sp$--algebras.
    
     \begin{equation} \label{diag zalg}
 \begin{tikzcd}
H\Z \arrow[rd] \arrow[dd,swap,"\varphi_Y"]
 &
 &
 \\
 &H\F_p \arrow[dr,"\varphi_X"]   
 \\
 U(X)=U(Y) & & U(X) \arrow[ll,"\simeq"] \arrow[ll,swap,"\psi"]
 \end{tikzcd}
 \end{equation}
 From the above diagram, by applying the $H\F_p$ homology functor we obtain the following diagram.
         \begin{equation} \label{diag zalg 2}
 \begin{tikzcd} 
 {H\F_p}_* H\Z \arrow[rd] \arrow[dd,swap,"{H\F_p}_* \varphi_Y"]
 &
 &
 \\
 &{H\F_p}_* H\F_p \arrow[dr,"{H\F_p}_* \varphi_X"]   
 \\
 {H\F_p}_* X \arrow[d]& & {H\F_p}_* X \arrow[ll,"\simeq"] \arrow[ll,swap,"{H\F_p}_* \psi"]
 \\
 X_*
 \end{tikzcd}
 \end{equation}
 Therefore, all the morphisms in this triangle preserve Dyer--Lashof operations. The bottom left arrow is induced by the $H\F_p$--module structure map of $X$. This is a morphism of commutative $H\F_p$--algebras, therefore the bottom left vertical arrow also preserves Dyer--Lashof operations operations. In conclusion, all the arrows in this diagram preserve Dyer--Lashof operations. 
    
    The composition of the vertical arrows on the left gives the map $f$ as chosen above. Therefore, $\overline{\tau}_1$ in ${H\F_p}_* H\Z$ is mapped to a non-zero element in $X_*$ by the composition of the vertical arrows. Because the triangle above commutes, if we travel $\overline{\tau}_1$ through the diagonal arrow to ${H\F_p}_* H\F_p$ and then to $X_*$, we see that $\overline{\tau}_1$ in ${H\F_p}_*  H\F_p$ must also be mapped to a non-zero element in $X_*$. Because  $\pi_1 X$ is trivial, $\tau_0$ in ${H\F_p}_* H\F_p$ is mapped to zero in $X_*$. However this, and the fact that $\overline{\tau}_1 = - \beta \mathrm{Q}^1 \tau_0$ imply that $\overline{\tau}_1$ in ${H\F_p}_* H\F_p$ is mapped to zero in $X_*$. This contradicts $f(\overline{\tau}_1) \neq 0$. Therefore, $Y$ and $X$ are not weakly equivalent as commutative $H\Z$--algebras. 
    
    For the case $H_{2p-2} X \neq 0$, we use an $f$ in 
    $\mathrm{Hom}_{\mathcal{R}\mhyphen alg} ({H\F_p}_*H\Z,X_*)$ that maps $\zeta_1$ to a non-zero element in $\pi_{2p-2}X$. Since $\mathrm{Q}^1 \tau_0 = -\zeta_1$, the rest of the argument follows similarly. 
    
    For $p=2$, we start with an $f$ in $\mathrm{Hom}_{\mathcal{R}\mhyphen alg} ({H\F_2}_*H\Z,X_*)$ that maps $\xi_1^2$ in ${H\F_2}_*H\Z$ to a non-zero element in $\pi_2 X$. Note that $\xi_1^2$ in ${H\F_2}_* H\Z$ is a free algebra generator and $\xi_1 \not \in {H\F_2}_* H\Z$. The arguments are similar but we do not use Dyer--Lashof operations in this case. Again considering Diagram \eqref{diag zalg 2}, $\xi_1^2$ in ${H\F_2}_* H\F_2$ is mapped to a non zero element in $\pi_2 X$ but since $\pi_1 X = 0$, $\xi_1$ is mapped to zero and this is a contradiction. Since we haven't used Dyer--Lashof operations, we may consider the underlying associative $H\Z$--algebras of $X$ and $Y$ and the above arguments still work, ie the associative $\Z$--DGAs corresponding to $X$ and $Y$ are non-trivially topologically equivalent.

    What is left to prove is the following lemma which says that the obstructions in the above setting for lifting $f$ to a map of commutative $\Sp$--algebras are zero. 
   \end{proof}

    \begin{lemma} \label{obstructions zero}
    In the setting of Theorem \ref{examples 2}, 
    \[\mathrm{Der}^t_{\mathcal{R}\mhyphen alg} ({H\F_p}_*H\Z,X^{S^{t-1}}_* )=0 \  \text{for} \ t \geq 2 \]
    \end{lemma}
    \begin{proof}
    We describe the odd prime case, the proof is similar for $p=2$. Let  $F_{\mathcal{R}}(\zeta_1,\overline{\tau}_1)$ denote the free unstable algebra over the Dyer--Lashof algebra generated by two elements whose degrees are the degrees of the corresponding generators in the dual Steenrod algebra. The free unstable algebra over the Dyer--Lashof algebra is described in Theorem \ref{BakerfreeDL}. The lowest degree generator of $F_{\mathcal{R}}(\zeta_1,\overline{\tau}_1)$ after $\zeta_1$ and $\overline{\tau}_1$ is $\beta \mathrm{Q}^p \zeta_1$ with degree $2p^2-3$. Note that ${H\F_p}_*H\Z$ has no free algebra generators in this degree, showing that ${H\F_p}_*H\Z$ cannot be the free unstable algebra over the Dyer--Lashof algebra generated by $\zeta_1$ and $\overline{\tau}_1$. However, also note that ${H\F_p}_*H\Z$ agrees with $F_{\mathcal{R}}(\zeta_1,\overline{\tau}_1)$ up to degree $2p^2-4$. Therefore, the morphism $F_{\mathcal{R}}(\zeta_1,\overline{\tau}_1)\to {H\F_p}_* H\Z$ which preserves Dyer--Lashof operations maps $\zeta_1$ to $\zeta_1$ and $\overline{\tau}_1$ to $\overline{\tau}_1$ is an isomorphism up to degree $2p^2-4$. 
    
    Furthermore, considering the free simplicial resolution of these objects as unstable algebras over the Dyer--Lashof algebra, we get a morphism
    \[F_{\mathcal{R}}^{\bullet +1}(F_{\mathcal{R}}(\zeta_1,\overline{\tau}_1)) \to F_{\mathcal{R}}^{\bullet +1}({H\F_p}_* H\Z) \]
    of simplicial unstable algebras over the Dyer--Lashof algebra. Note that up to degree $n>0$, $F_\mathcal{R}(M)$ only depends on the part of the vector space $M$ up to degree $n$. Therefore, the morphism above is an isomorphism up to degree $2p^2-4$ at each simplicial degree.
    
    Let the functor 
    \[\mathrm{Der}_{\mathcal{R}\mhyphen alg} (-,X^{S^{t-1}}_*) = \mathrm{Map}_{\mathcal{R} \mhyphen alg \downarrow X_*}(-,X^{S^{t-1}}_*)\] denote the degree preserving derivations that preserve Dyer--Lashof operations.  Since $X^{S^{t-1}}_*$ is concentrated in degree $2p^2-4$ and below, this functor depends only on the input up to degree  $2p^2-4$. Therefore the morphism of simplicial sets above induces an isomorphism 
    \[\mathrm{Der}_{\mathcal{R}\mhyphen alg} (F_{\mathcal{R}}^{\bullet +1}({H\F_p}_* H\Z) ,X^{S^{t-1}}_*) \cong \mathrm{Der}_{\mathcal{R}\mhyphen alg} (F_{\mathcal{R}}^{\bullet +1}(F_{\mathcal{R}}(\zeta_1,\overline{\tau}_1)),X^{S^{t-1}}_*).  \] 
    In cohomology, this induces the isomorphism  \[\mathrm{Der}^t_{\mathcal{R}\mhyphen alg} ({H\F_p}_*H\Z,X^{S^{t-1}}_*) \cong \\ \mathrm{Der}^t_{\mathcal{R}\mhyphen alg} (F_{\mathcal{R}}(\zeta_1,\overline{\tau}_1),X^{S^{t-1}}_*) = 0 \mathrm{\ for}\  t>0.\]
    The last equality follows because Andr{\'e}--Quillen cohomology of a free object is trivial above degree zero.

    \end{proof}

\section{Proof of Theorem \ref{nonexistence 2}} \label{section nonexist}

To prove Theorem \ref{nonexistence 2}, we need to show that two co-connective commutative $H\Z/(m)$--algebras that are weakly equivalent as commutative $\Sp$--algebras are also weakly equivalent as commutative $H\Z/(m)$--algebras for any $m \in \Z$ with $m\neq \pm 1$.

Since the category of commutative $H\Z/(m)$--algebras is the category of commutative $\Sp$--algebras under $H\Z/(m)$, for our purpose, it is natural to consider the  homotopy class of commutative $\Sp$--algebra maps from $H\Z/(m)$ to a co-connective $H\Z/(m)$--algebra $X$. We omit the forgetful functor to commutative $\Sp$--algebras and denote this by $\pi_{0} \mathrm{map}_{\Sp \mhyphen cAlg}(H\Z/(m),X)$. We show in Proposition \ref{calculation 1} that there is a unique homotopy class of maps in $\mathrm{map}_{\Sp \mhyphen cAlg}(H\Z/(m),X)$. The proof of Theorem \ref{nonexistence 2} is based on this fact. 

\begin{proof}[Proof of Theorem \ref{nonexistence 2}]
Let $X$ and $Y$ be co-connective commutative \\ $H\Z/(m)$--algebras that are weakly equivalent as commutative $\Sp$--algebras. We assume $X$ and $Y$ are cofibrant and fibrant as commutative $H\Z/(m)$--algebras. Recall that cofibrations, fibrations and weak equivalences of commutative $H\Z/(m)$--algebras are precisely those of commutative $\Sp$--algebras. Therefore $X$ and $Y$ are also fibrant as commutative $\Sp$--algebras. Furthermore, they are also cofibrant commutative $\Sp$--algebras because the initial map $\Sp \to U(X)$ factors as a composition of two commutative $\Sp$--algebra cofibrations as$\begin{tikzcd}[sep=small]\Sp \arrow[r,hookrightarrow]& H\Z/(m) \arrow[r,hookrightarrow]& U(X),\end{tikzcd}$where $U$ denotes the forgetful functor to commutative $\Sp$--algebras. 

Since $X$ and $Y$ are weakly equivalent as commutative $\Sp$--algebras and they are cofibrant and fibrant as commutative $\Sp$--algebras, there is a weak equivalence of commutative $\Sp$--algebras$\begin{tikzcd}[sep=small] \psi  \co U(X) \arrow[r,"\sim"]&U(Y).\end{tikzcd}${L}et $\varphi_X  \co H\Z/(m) \to U(X)$ and $\varphi_Y \co H\Z/(m) \to U(Y)$ denote the commutative $\Sp$--algebra maps that are the $H\Z/(m)$ structre maps of $X$ and $Y$ respectively. Since $\psi$ is only a commutative $\Sp$--algebra map, it may not preserve the $H\Z/(m)$ structure, ie $\psi \circ \varphi_X$ is not necessarily equal to $\varphi_Y$. 

Let $Y^\prime$ be the commutative $H\Z/(m)$--algebra whose underlying commutative $\Sp$--algebra is $U(Y)$ and whose $H\Z/(m)$ structure map is $\psi \circ \varphi_X$. With this $H\Z/(m)$ structure of $Y^\prime$, $\psi$ becomes a weak equivalence of commutative $H\Z/(m)$--algebras from $X$ to $Y^\prime$. Therefore it is sufficient to show that $Y^\prime$ and $Y$ are weakly equivalent as $H\Z/(m)$--algebras. 

By Proposition \ref{calculation 1}, $\pi_{0} \mathrm{map}_{\Sp \mhyphen cAlg}(H\Z/(m),Y)= \{*\}$. Therefore, $\varphi_Y$ and $\psi \circ \varphi_X$, the structure maps of $Y$ and $Y^\prime$ respectively, are homotopic. A homotopy between $\varphi_Y$ and $\psi \circ \varphi_X$ is given by the following diagram.

 \begin{equation} \label{diag Fp theorem}
 \begin{tikzcd}
 & & U(Y)
 &
 &
 \\
 H\Z/(m) \arrow[urr,bend left,"\varphi_Y"] \arrow[drr,bend right,swap, "\psi \circ \varphi_X"] \arrow[r,"f"]    & U(Y)^I \arrow[ur,swap,"p_1"] \arrow[ur,"\simeq"] \arrow[dr,swap,"\simeq"] \arrow[dr,"p_2"]
 \\
 & & U(Y)= U(Y^{'})
 \end{tikzcd}
 \end{equation}
 Here, $U(Y)^I$ denotes a path object of $U(Y)$. This is a diagram in commutative $\Sp$--algebras. However, if we give $U(Y)^I$ a commutative $H\Z/(m)$--algebra structure using $f$ and call this commutative $H\Z/(m)$--algebra $Z$, $p_1$ becomes a weak equivalence of commutative $H\Z/(m)$--algebras from $Z$ to $Y$ and $p_2$ becomes a weak equivalence of commutative $H\Z/(m)$--algebras from $Z$ to $Y^\prime$. Therefore, $Y$ and $Y^\prime$ are weakly equivalent commutative $H\Z/(m)$--algebras and so are $Y$ and $X$.


\end{proof}

What is left to prove is the following proposition.

\begin{proposition} \label{calculation 1}
For a co-connective commutative $H\Z/(m)$--algebra $X$, the mapping space $\mathrm{map}_{\Sp \mhyphen cAlg}(H\Z/(m),X)$ is contractible. 
\end{proposition}

We use the obstruction spectral sequence to show that all the homotopy groups of this mapping space are trivial. However, since we work over a general $\mathbb{Z}/(m)$ where $m$ may not be a prime, we do not have a description of the $E_2$ page of the spectral sequence as in Theorem \ref{obstruction}. It turns out that it is sufficient to consider the $E_1$ page only. For this purpose, we use the spectral sequence in Theorem A of \cite{johnson2014lifting} which we can do due to Corollary 4.13 of \cite{johnson2014lifting}. Because $X$ is a commutative $H\Z/(m)$--algebra, there is a map of commutative $\Sp$--algebras $H\Z/(m) \to U(X)$ that serves as a base point for this spectral sequence. 

In this setting, the $E_1$ page of this spectral sequence is given by 
\[E_1^{s,t} = \pi_{t} \mathrm{map}_{\Sp \mhyphen cAlg}(\Pp^{s +1}_{\Sp}(H\Z/(m)),X),\]
where the homotopy groups are calculated at the given base point, see the proof of Theorem A of \cite{johnson2014lifting}. This spectral sequence abuts to
\[\pi_{t-s} \mathrm{map}_{\Sp \mhyphen cAlg}(H\Z/(m),X).\]
The free commutative $\Sp$--algebra functor induces a monad in the homotopy category of $\Sp$--modules and let $h\Pp_{\Sp}$ denote this monad. We use that the $E_2^{0,0}$ term of this spectral sequence is given by 
\[E_2^{0,0}= \mathrm{Hom}_{h\Pp_{\Sp} \mhyphen alg}(H\Z/(m),X)\]
which denotes morphisms of algebras over $h\Pp_{\Sp}$ in $\Sp$--modules. Note that $E_2^{0,0}$ term of this spectral sequence is just a set.  

\begin{proof}
We will show that in the above spectral sequence, $E_1^{s,t} = 0$ for $t>0$ and $E_2^{0,0} = pt$. This is sufficient to show that the homotopy groups of the mapping space are trivial. 

We start by showing that $E_2^{0,0} = 0$. We have the following isomorphisms for $E_1^{0,0}$ that we explain below.
\begin{equation*} \label{eq_1}
\begin{split}
E_1^{0,0} & = \pi_{0} \mathrm{map}_{\Sp \mhyphen cAlg}(\Pp_{\Sp}(H\Z/(m)),X) \\
 & \cong \pi_{0} \mathrm{map}_{\Sp \mhyphen mod}(H\Z/(m),X) \\
 & \cong \mathrm{Hom}_{
 \Z \mhyphen mod}(H\Z_* H\Z/(m),X_*) \\
 & \cong \mathrm{Hom}_{\Z \mhyphen mod}(\Z/(m),X_0) \\
\end{split}
\end{equation*}
 The first isomorphism follows by adjunction. For the second isomorphism, we use the universal coefficient spectral sequence of Theorem 4.5 in Chapter \rom{4} of \cite{elmendorf2007rings} with respect to the homology theory $H\Z_*$. This works because $X$ is an $H\Z$--module by forgetting the $H\Z/(m)$--module structure through the map $H\Z \to H\Z/(m)$. For this spectral sequence, we have
 \[E_2^{p,q} = \mathrm{Ext}_{\Z}^{p,-q}(H\Z_*H\Z/(m),X_*)\]
 \[d^r_{p,q} \co E_2^{p,q} \to E_2^{p+r,q-r+1}\]
 where $p$ denotes the cohomological degree and $-q$ denotes the internal degree, particularly it denotes the Ext groups calculated by considering the maps that increase the degree by $-q$. This spectral sequence abuts to 
 \[\pi_{-(p+q)} \mathrm{map}_{\Sp \mhyphen mod}(H\Z/(m),X).\]

 Since $\Z$ has global dimension $1$, $E_2^{p,q}=0$ for $p>1$ and therefore all differentials are zero and the only terms that contribute to $\pi_0$ are $E_2^{0,0}$ and $E_2^{1,-1}$. Note that by Proposition 1.2 in Chapter \rom{4} of \cite{elmendorf2007rings}, $H\Z_*H\Z/(m)$ is connective. Therefore there is a projective resolution of $H\Z_*H\Z/(m)$ that is connective in each resolution degree. From such a resolution, there are no maps of degree $1$ to the co-connective object $X_*$. Hence, $E_2^{1,-1}=0$. This proves the second isomorphism. By the Tor spectral sequence of Theorem 4.1 in Chapter \rom{4} of \cite{elmendorf2007rings}, $H\Z_0H\Z/(m) \cong \Z/(m)$. The third isomorphism follows from this because $X$ is co-connective.

By the description of the $E_2^{0,0}$ above, a morphism in $E_1^{0,0}$ that lifts to $E_2^{0,0}$ should preserve the multiplicative identity. Since $\Z/(m)$ is generated as an abelian group by $1$, there is only one such morphism in $E_1^{0,0}\cong \mathrm{Hom}_{\Z \mhyphen mod}(\Z/(m),X_0)$. Furthermore, we know that this morphism lifts to the $E_2^{0,0}$ term because it is represented by a morphism which is an actual commutative $\Sp$--algebra map $H\Z/(m) \to X$ which is our base point. In conclusion, $E_2^{0,0} = pt$.  

Now we will show that $E_1^{s,t}= 0$ for $t>0$. Again by adjunction, we have 
\begin{equation*} \label{eq1}
\begin{split}
E_1^{s,t} &= \pi_{t} \mathrm{map}_{\Sp \mhyphen cAlg}(\Pp_{\Sp}^{s+1}(H\Z/(m)),X) \\
&\cong \pi_{t} \mathrm{map}_{\Sp \mhyphen mod}(\Pp_{\Sp}^{s}(H\Z/(m)),X).\\
\end{split}
\end{equation*}

There is a spectral sequence for calculating homotopy groups of the homotopy orbit of a spectrum with an action of a group $G$. 
\[ H_p(G,\pi_qY) \Rightarrow \pi_{p+q}Y_{hG}\]
Using this spectral sequence it is clear that the homotopy orbit spectrum of a connective spectrum is connective. Therefore $\Pp_{\Sp}(Y)$ is connective when $Y$ is because $\Pp_{\Sp}(Y)$ is wedges of homotopy orbits of $Y^{\wedge n}$ with respect to the action of the symmetric group. This means that $\Pp_{\Sp}^{s}(H\Z/(m))$ is connective. Since $X$ is co-connective by hypothesis, we have   
\[\pi_{t} \mathrm{map}_{\Sp \mhyphen mod}(\Pp_{\Sp}^{s}(H\Z/(m)),X) \cong 0 \ \text{for} \ t>0\]
by Proposition 1.4 in Chapter \rom{4} of \cite{elmendorf2007rings}.
 
\end{proof}    

\section{Proof of Theorem \ref{nonexistence 3}} \label{section dl}

In this section, we prove the following theorem.
\begin{theorem}
\label{thm DL}
Let $X$ and $Y$ be $H_\infty$ $H\F_p$--algebras with trivial first homotopy groups. If $X$ and $Y$ are equivalent as $H_\infty$ $\Sp$--algebras, then they are equivalent as $H_\infty$ $H\F_p$--algebras.
\end{theorem}

This is a slightly stronger result than Theorem \ref{nonexistence 3}. If two $E_\infty$ DGAs are $E_\infty$ topologically equivalent, then the corresponding ring spectra are commutative $\Sp$--algebra equivalent and therefore $H_\infty$ $\Sp$--algebra equivalent. Therefore, Theorem \ref{nonexistence 3} is a corollary of Theorem \ref{thm DL}. 

\begin{remark}
As mentioned in Remark \ref{remark Tyler}, one of the intermediate results of \cite{lawson2015note}, Proposition 5, states that Theorem \ref{thm DL} is still true for $H_\infty$ $H\F_p$--algebras with non-trivial first homology and this contradicts Example \ref{example Fp}. The proof of Proposition 5 of \cite{lawson2015note} ends by stating that the canonical map 
\begin{equation}
[\Pp_{H\F_p}(M),M]_{H\F_p\mhyphen mod} \to [\Pp_{\Sp}(M),M]_{\Sp \mhyphen mod}
\end{equation}
between homotopy classes of maps in $H\F_p$--modules to $\Sp$--modules is injective where $M$ is an $H\F_p$--module. This says that $H_\infty$ $H\F_p$--algebra structure maps forget injectively to $H_\infty$ $\Sp$--algebra structure maps. However, this does not imply the desired result since one needs to consider $H_\infty$ $H\F_p$--equivalences and $H_\infty$ $\Sp$--equivalences between different $H_\infty$ $H\F_p$--algebra and $H_\infty$ $\Sp$--algebra structures on $M$. 
  
\end{remark}

In the proof of Theorem \ref{thm DL}, we use the following facts about $H_\infty$ algebras which can be derived using the results of \cite{brunerh}. In the items below, $X$ denotes an $H_\infty$ $H\F_p$--algebra. 

\begin{enumerate}
\item A morphism of $H_\infty$ $\Sp$--algebras induces a map of rings in the homotopy groups.

\item The structure map $\mu_X \co  H\F_p \wedge X \to X$ induced by the $H\F_p$--module structure on $X$ is a map of $H_\infty$ $H\F_p$--algebras. Therefore this map preserves Dyer--Lashof operations on the homotopy ring.

\item There is an equivalence $H\F_p\wedge X \cong (H\F_p \wedge H\F_p)\wedge_{H\F_p} X$. Using this, we obtain the identification $\pi_* (H\F_p \wedge X) \cong \mathcal{A}_* \otimes_{\F_p} X_*$. Note that the Dyer--Lashof operations on $\mathcal{A}_* \otimes_{\F_p} X_*$ are not those of the tensor product because the $H\F_p$ structure on $H\F_p \wedge X$ is given by multiplication with the $H\F_p$ factor on the left. With this identification, ${\mu_X}_*$ is given by ${\mu_X}_*(a\otimes x) = ax$ if $a \in \mathcal{A}_0= \F_p$ and ${\mu_X}_*(a\otimes x) = 0$ if $a \in \mathcal{A}_i$ for $i>0$.
 
\item The unit map $\eta_X \co \Sp \wedge X \to H\F_p\wedge X$ satisfies $\mu_X\circ\eta_X= id$. However, $\eta_X$ is only a map of $H_\infty$ $\Sp$--algebras and it may not preserve Dyer--Lashof operations in the homotopy ring. By the identification of $\pi_*(H\F_p \wedge X)$ above, the morphism induced by $\eta_X$ on the homotopy ring is given by ${\eta_X}_* (x) = 1\otimes x$.
\end{enumerate}


\begin{proof}[Proof of Theorem \ref{thm DL}]
    Let $\varphi \co X\to Y$ be an equivalence of $H_\infty$ $\Sp$--algebras. This implies that $\varphi_*$ is an isomoprhism of rings. We will show that $\varphi$ induces an equivalence of $H_\infty$ $H\F_p$--algebras by showing that $\varphi_*$ preserves Dyer--Lashof operations. This is sufficient because an $H_\infty$ $H\F_p$--algebra equivalence type is determined by the isomorphism class of its homotopy ring as an algebra over the  Dyer--Lashof algebra, see Theorem 4 in \cite{lawson2015note} and the discussion after it.
    
    We have the following diagram.
\begin{equation} \label{diag 1}
 \begin{tikzcd}
 X \arrow [r,"\varphi"]
 \arrow[d,"\eta_X"]
 &
 Y
 \arrow[d,"\eta_Y"]
 \
 \\
 \
 H\F_p \wedge X
 \arrow[r,"id \wedge \varphi"]
 \arrow[d,"\mu_X"]
 &H\F_p \wedge Y
 \arrow[d,"\mu_Y"]
 \
 \\
 \
 X
 &Y
 \end{tikzcd}
\end{equation}
Applying the homotopy functor to this diagram produces the following. 
\begin{equation} \label{diag 2}
 \begin{tikzcd}
 X_* \arrow [r,"\varphi_*"]
 \arrow[d,"{\eta_X}_*"]
 &
 Y_*
 \arrow[d,"{\eta_Y}_*"]
 \
 \\
 \
 \mathcal{A}_* \otimes_{\F_p} X_*
 \arrow[r,"\psi"]
 \arrow[d,"{\mu_X}_*"]
 &\mathcal{A}_* \otimes_{\F_p} Y_*
 \arrow[d,"{\mu_Y}_*"]
 \
 \\
 \
 X_* \arrow[r,"\varphi_*"]
 &Y_*
 \end{tikzcd}
\end{equation}

The middle horizontal morphism $\psi$ is the morphism induced on the homotopy groups by $id \wedge \varphi$. Because we do not assume $\varphi$ to be a map of $H_\infty$ $H\F_p$--algebras, $\psi$ may not be induced by two morphisms on the tensor factors. However, $\psi$ preserves Dyer--Lashof operations because it is the morphism in $H\F_p$ homology induced by $\varphi$. 

The top square in Diagram \eqref{diag 2} commutes because it is induced by the commutative square in Diagram \eqref{diag 1}. Although the bottom square is not induced by a commuting square, we show that it also commutes. For this purpose we need to know more about the Dyer--Lashof operations on $\mathcal{A}_* \otimes_{\F_p} X_*$. 

We have the following map
\[ H\F_p \wedge H\F_p \cong (H\F_p \wedge H\F_p) \wedge_{H\F_p} H\F_p \to (H\F_p \wedge H\F_p) \wedge_{H\F_p} X \cong H\F_p \wedge X\]
induced by the map of $H_\infty$ $H\F_p$--algebras $H\F_p \to X$. This is a map of $H_\infty$ $H\F_p$--algebras when the $H\F_p$ multiplication on $H\F_p \wedge H\F_p$ is given by that of the $H\F_p$ factor on the left. Therefore the morphism $\mathcal{A}_* \to \mathcal{A}_* \otimes_{\F_p} X_*$ induced on the homotopy groups preserves Dyer--Lashof operations. This says that on $\mathcal{A}_*\otimes_{\F_p} \{1\} \subseteq \mathcal{A}_*\otimes_{\F_p} X_*$, Dyer--Lashof operations are given by the ones on the dual Steenrod algebra ie $\mathrm{Q}^s (a \otimes 1)= (\mathrm{Q}^s a) \otimes 1$.

Now we show that the bottom square in Diagram \eqref{diag 2} commutes. We first show this for elements of the form $a \otimes x \in \mathcal{A}_*\otimes_{\F_p} X_*$ with $\lvert a \rvert >0$. By the description of ${\mu_X}_*$ in the paragraph before this proof, we have ${\mu_X}_*(a\otimes x)= 0$ and therefore $\varphi_* \circ {\mu_X}_* (a \otimes x)= 0$. Therefore our goal is to show that ${\mu_Y}_* \circ \psi (a \otimes x) = 0$. Let $\tau_0$ denote the degree $1$ element in $\mathcal{A}_*$ that generates it as an algebra over the Dyer--Lashof algebra (this element is called $\xi_1$ for $p=2$). Because $\pi_1(Y) = 0$, ${\mu_Y}_* \circ \psi (\tau_0 \otimes 1) = 0$. Since ${\mu_Y}_* \circ \psi$ is a morphism of rings that preserves Dyer--Lashof operations, 
${\mu_Y}_* \circ \psi (a \otimes 1) = 0$ whenever $\lvert a \rvert>0$. Therefore when $\lvert a \rvert >0$,
\[{\mu_Y}_* \circ \psi (a \otimes x) = ({\mu_Y}_* \circ \psi (a \otimes 1)) \cdot ({\mu_Y}_* \circ \psi (1 \otimes x)) = 0 .\]

After this, we just need to show that the bottom square in Diagram \eqref{diag 2} commutes for elements in $\mathcal{A}_* \otimes_{\F_p} X_*$ of the form $a \otimes x$ where $a \in \mathcal{A}_0 = \F_p$. Clearly, it is sufficient to work only with the elements of the form $1 \otimes x$. By the description of ${\mu_X}_*$ in the paragraph before this proof, we have $\varphi_* \circ {\mu_X}_*(1\otimes x ) = \varphi_*(x)$. Therefore, our goal is to show that ${\mu_Y}_*\circ \psi(1 \otimes x) = \varphi_*(x)$.
Because the top square in Diagram \eqref{diag 2} commutes, we deduce that 
\begin{equation} \label{eq 1}
 \psi (1\otimes x) = \psi({\eta_X}_*(x)) = {\eta_Y}_*(\varphi_*(x)) = 1\otimes \varphi_*(x).
\end{equation}
Using this, we obtain what we wanted to show: \[{\mu_Y}_* \circ \psi (1 \otimes x)= {\mu_Y}_* (1 \otimes \varphi_*(x)) = \varphi_*(x).\]
At this point, we know that the bottom square in Diagram \eqref{diag 2} commutes and we are ready to show that $\varphi_*$ preserves Dyer--Lashof operations. Given $x \in X_*$, we have 
\[ \varphi_*(\mathrm{Q}^s x) = \varphi_*(\mathrm{Q}^s {\mu_X}_*(1 \otimes x)) = \varphi_* \circ {\mu_X}_*(\mathrm{Q}^s (1 \otimes x)) = {\mu_Y}_* \circ \psi (\mathrm{Q}^s (1 \otimes x)).\]
Therefore we need to show that ${\mu_Y}_* \circ \psi (\mathrm{Q}^s (1 \otimes x)) = \mathrm{Q}^s \varphi_*(x)$. This is given by the following chain of equalities
\[{\mu_Y}_* \circ \psi (\mathrm{Q}^s (1 \otimes x)) = \mathrm{Q}^s{\mu_Y}_* (\psi (1 \otimes x)) = \mathrm{Q}^s{\mu_Y}_*(1\otimes \varphi_*(x)) = \mathrm{Q}^s\varphi_*(x).\]
The first equality follows because both $\psi$ and ${\mu_Y}_*$ preserve Dyer--Lashof operations and the second equiality follows by Equation \eqref{eq 1}. Since $\varphi_*$ preserves Dyer--Lashof operations, it induces an isomorphism between $X_*$ and $Y_*$ as algebras over the Dyer--Lashof algebra and therefore $X$ and $Y$ are equivalent as $H_\infty$ $H\F_p$--algebras.

\end{proof}
\appendix
\section{Previous examples} 
\label{sec previous}
  In Section \ref{sec previous examples}, we discuss the first class of examples of non-trivial topological equivalences provided in \cite{dugger2007topological}. These examples rely on the classification of Postnikov extensions of ring spectra developed in \cite{dugger2006postnikov}. In this section, we point out a mistake in the construction of these examples and provide a correction which recovers the classification of quasi-isomorphism classes of $\Z$--DGAs with homology ring $\Lambda_{\F_p}(x_n)$ for $n>0$, ie the exterior algebra with a single generator in a positive degree.

Let $R$ be a connective commutative ring spectrum. We first explain the classification of Postnikov extensions of connective (trivial negative homotopy groups) $R$--algebras developed in \cite{dugger2006postnikov}. For a connective $R$--algebra $X$, the $n$th Postnikov section of $X$ is a map of $R$--algebras $X\to P_nX$ which induces an isomorphism on $\pi_i(X) \to \pi_i (P_nX)$ for $i\leq n$ and with $\pi_i (P_nX) = 0$ for $i>n$. Given a connective $R$--algebra $Y$ with $P_{n-1}Y \simeq Y$ and a $\pi_0(Y)$--bimodule $M$, a Postnikov extension of $Y$ of type $(M,n)$ is a map of $R$--algebras $X \to Y$ which satisfies the following properties:
\begin{enumerate}
    \item   $\pi_i (X) = 0$ for $i>n$
    \item $\pi_i(X) \to \pi_i(Y)$ is an isomorphism for $i < n$
    \item There is an isomorphism of $\pi_0(X)$--bimodules $\pi_n(X) \cong M$ where $\pi_0(X)$--bimodule structure structure of $M$ is obtained by the map $\pi_0(X) \to \pi_0(Y)$.
\end{enumerate}

The moduli space of Postnikov extensions of $Y$ of type $(M,n)$, denoted by $\mathcal{M}_R(Y+M,n)$, is defined to be the category whose objects are Postnikov extensions of $Y$ of type $(M,n)$, and a morphism between two extensions $X_1 \to Y$ to $X_2 \to Y$ is a weak equivalence$\begin{tikzcd}[sep=small] X_1 \arrow[r,"\sim"]&X_2\end{tikzcd}$for which the following triangle commutes. 
 \begin{equation*}
 \begin{tikzcd}
 X_1 \arrow [rr,"\simeq"]
 \arrow[dr,swap,]
 &
 & X_2
 \arrow[dl,swap,]
 \\
 &
 Y 
 \end{tikzcd}
 \end{equation*}

The main result of \cite{dugger2006postnikov} is a classification of these Postnikov extensions in terms of topological Hochschild cohomology.
\begin{theorem} \label{DS 1}
\cite{dugger2006postnikov} Assuming $X$ is cofibrant as an $R$--module, the following is a bijection.
\begin{equation*}
\pi_0 \mathcal{M}_R(X+M,n) \cong \mathrm{THH}_R^{n+2} (X,M) / \mathrm{Aut}(M)
\end{equation*}
\end{theorem}

This result is used in \cite{dugger2007topological} to classify weak equivalence classes of $\Z$--DGAs with homology ring  $\Lambda_{\F_p} (x_n)$, the exterior algebra over $\F_p$ with a single generator in degree $n$, for $n>0$. Any such DGA is a Postnikov extension of $\F_p$ of type $(\F_p,n)$. By the Quillen equivalence of  $\Z$--DGAs and $H\Z$--algebras, this is the same as classifying $H\Z$--algebras with homotopy ring $\Lambda_{\F_p} (x_n)$ and such $H\Z$--algebras are Postnikov extensions of $H\F_p$ of type $(H\F_p,n)$ in $H\Z$--algebras. 

At this point we note the piece of explanation that is missing in \cite{dugger2007topological} about this classification. In \cite{dugger2007topological} it is claimed that weak equivalence classes of $H\Z$--algebras whose homotopy ring are $\Lambda_{\F_p} (x_n)$ are classified by $\pi_0 \mathcal{M}_{H\Z}(H\F_p+H\F_p,n)$. However in $\mathcal{M}_{H\Z}(H\F_p+H\F_p,n)$, the morphisms are weak equivalences of Postnikov extensions, ie for two Postnikov extensions of $H\F_p$ of type $(H\F_p,n)$: $X_1 \to H\F_p$ and $X_2\to{H\F_p}$, a morphism in $\mathcal{M}_{H\Z}(H\F_p+H\F_p,n)$ is a weak equivalence of $H\Z$--algebras$\begin{tikzcd}[sep=small] X_1 \arrow[r,"\sim"]&X_2\end{tikzcd}$for which the triangle above commutes. In the classification we are concerned with here, a morphism of two Postnikov extensions is just a weak equivalence of $H\Z$--algebras$\begin{tikzcd}[sep=small] X_1 \arrow[r,"\sim"]&X_2.\end{tikzcd}$ In general one should not expect these two classifications to be the same. However, in this case we will show that they are actually the same. Unfortunately we cannot use  simple point set arguments to prove this, even if we work with $\Z$--DGAs instead of $H\Z$--algebras, because one needs to use a cofibrant and fibrant $H\Z$--algebra model of $H\F_p$ and also because we need the same result over $\Sp$--algebras not only for $H\Z$--algebras. 

We prove that these two classifications are the same by first showing that there is a unique homotopy class of $H\Z$--algebra maps from $X$ to $H\F_p$. We prove this in Lemma \ref{lemma unique map} by using the obstruction theory of the Hopkins--Miller theorem. Using this fact, we prove that $\pi_0 \mathcal{M}_{H\Z}(H\F_p+H\F_p,n)$ actually classifies weak equivalence classes of $H\Z$--algebras with homotopy ring $\Lambda_{\F_p} (x_n)$.  

\begin{proposition}
\label{classification 1}
The set $\pi_0 \mathcal{M}_{H\Z}(H\F_p+H\F_p,n)$ is in bijective correspondence with weak equivalence classes of $H\Z$--algebras with homotopy ring $\Lambda_{\F_p} (x_n)$ for $n>0$. This statement holds for $\Sp$--algebras too. Namely, the set $\pi_0\mathcal{M}_{\Sp}(H\F_p+H\F_p,n)$ is in bijective correspondence with weak equivalence classes of $\Sp$--algebras with homotopy ring $\Lambda_{\F_p} (x_n)$ for $n>0$.
\end{proposition}
\begin{proof}
We only prove the statement for $H\Z$--algebras. The proof is similar for $\Sp$--algebras, the only important difference is that one uses the second part of Lemma \ref{lemma unique map} instead of the first part.

Since up to weak equivalence there is a unique $H\Z$--algebra with homotopy $\F_p$ concentrated at degree zero, every $H\Z$--algebra with homotopy ring $\Lambda_{\F_p} (x_n)$ is a Postnikov extension of $H\F_p$ of type $(H\F_p,n)$. Given two such Postnikov extensions:  $\varphi_1 \co X_1 \to H\F_p$ and $\varphi_2\co X_2 \to H\F_p$, if these extensions are weakly equivalent in $\mathcal{M}_{H\Z}(H\F_p+H\F_p,n)$ then they are clearly weakly equivalent as $H\Z$--algebras. We need to show that when these Postnikov extensions are weakly equivalent as $H\Z$--algebras, they are also weakly equivalent in $\mathcal{M}_{H\Z}(H\F_p+H\F_p,n)$. 


Let $X_1$ and $X_2$ be weakly equivalent as $H\Z$--algebras, we show that $\varphi_1$ and $\varphi_2$ are weakly equivalent in $\mathcal{M}_{H\Z}(H\F_p+H\F_p,n)$. We assume that $X_1$ and $X_2$ are both fibrant and cofibrant and $H\F_p$ is fibrant as $H\Z$--algebras. 

Because $X_1$ and $X_2$ are weakly equivalent as $H\Z$--algebras and because $X_1$ is cofibrant and $X_2$ is fibrant, there is a weak equivalence of $H\Z$--algebras$\begin{tikzcd}[sep=small]\psi \co X_1 \arrow[r,"\sim"]&X_2.\end{tikzcd}$ Using this weak equivalence we define another Postnikov extension of $H\F_p$ of type $(H\F_p,n)$ which is the composite $\varphi_2 \circ \psi \co X_1 \to H\F_p$. This Postnikov extension, $\varphi_2 \circ \psi$, is weakly equivalent to $\varphi_2$ in $\mathcal{M}_{H\Z}(H\F_p+H\F_p,n)$ through $\psi$. Therefore it is sufficient to show that $\varphi_2 \circ \psi$ is weakly equivalent to $\varphi_1$ in $\mathcal{M}_{H\Z}(H\F_p+H\F_p,n)$. By Lemma \ref{lemma unique map}, there is a unique homotopy class of maps from $X_1$ to $H\F_p$. Therefore, $\varphi_1$ and $\varphi_2 \circ \psi$ are homotopic. We have the following diagram in $H\Z$--algebras which corresponds to a homotopy between these maps where $X_1 \wedge I$ is a path object of $X_1$. 

\begin{equation*}
 \begin{tikzcd}
 X_1 \arrow [dr,"\simeq"]
 \arrow[ddr,swap,"\varphi_1"]
 &
 & X_1
 \arrow[dl,swap,"\simeq"]
 \arrow[ddl,"\varphi_2 \circ \psi"]
 \\
 &
 X_1 \wedge I
 \arrow[d]
 \\
 &H\F_p
 \end{tikzcd}
 \end{equation*}

The map $X_1 \wedge I \to H\F_p$ in the above diagram is also a Postnikov extension of type $(H\F_p,n)$ of $H\F_p$ because it factors $\varphi_1$ by a weak equivalence. Therefore the above diagram gives a zig-zag of weak equivalences between $\varphi_1$ and $\varphi_2 \circ \psi$ in $\mathcal{M}_{H\Z}(H\F_p+H\F_p,n)$. This shows that $\varphi_1$ and $\varphi_2$ are weakly equivalent in $\mathcal{M}_{H\Z}(H\F_p+H\F_p,n)$.

\end{proof}

At this point, we are ready to provide the classification of weak equivalence classes of $H\Z$--algebras with homotopy ring $\Lambda_{\F_p} (x_n)$ for $n>0$ and hence, quasi-isomorphism classes of $\Z$--DGAs with homology ring $\Lambda_{\F_p} (x_n)$ for $n>0$. By Proposition \ref{classification 1} and Theorem \ref{DS 1}, This is given by $\mathrm{THH}_{H\Z}^{n+2} (H\F_p,H\F_p) / \mathrm{Aut}(\F_p)$. As described in Example 3.15 of \cite{dugger2007topological}, $\mathrm{THH}_{H\Z}^{*} (H\F_p,H\F_p) \cong \F_p[\sigma_2]$, a polynomial algebra with a generator in degree 2 (with cohomological grading). Calculating the quotient of $\F_p[\sigma_2]$ by the multiplicative action of $\F_p$, one obtains the following classification: for odd $n>0$, there is a unique $H\Z$--algebra with homotopy ring $\Lambda_{\F_p} (x_n)$ and for even $n>0$, there are exactly two non-weakly equivalent $H\Z$--algebras whose homotopy ring is $\Lambda_{\F_p} (x_n)$. For $n=0$, this classification says that there are two Postnikov extensions of $H\F_p$ of type $(H\F_p,0)$ and these are $H \Z / p^2$ and $H \Lambda_{\F_p} (x_0)$.

Similarly, weak equivalence classes of $\Sp$--algebras with homotopy ring $\Lambda_{\F_p} (x_n)$ are given by $\mathrm{THH}_{\Sp}^{n+2} (H\F_p,H\F_p) / \mathrm{Aut}(\F_p)$. In this case,  
 $\mathrm{THH}_{\Sp}^{*} (H\F_p,H\F_p) \cong \Gamma[\alpha_2]$ as rings where $\Gamma[\alpha_2]$ is the divided polynomial algebra on a generator of degree 2 which is isomorphic to $\F_p[\sigma_2]$ as an  $\F_p$--module, see Theorem 13.4.15 in \cite{loday1988cyclic}. Therefore we get a similar classification result: there are exactly two non weakly equivalent $\Sp$--algebras with homotopy ring $\Lambda_{\F_p} (x_n)$ for odd $n>0$ and there is only one for even $n>0$. 

What we are really interested in here is deciding which of these non-weakly equivalent $H\Z$--algebras are weakly equivalent as $\Sp$--algebras. For this, one considers the the map
\begin{equation*}
\mathrm{THH}_{H\Z}^{n+2} (H\F_p,H\F_p) \to \mathrm{THH}_{\Sp}^{n+2} (H\F_p,H\F_p)
\end{equation*}
induced by the forgetful functor from $H\Z$--algebras to $\Sp$--algebras. This corresponds to a morphism of rings $\varphi \co \F_p [\sigma_2] \to \Gamma[\alpha_2]$. Where $\varphi$ maps $\sigma_2$ to $\alpha_2$ because $H \Z / p^2$ and $H \Lambda_{\F_p} (x_0)$ are non-weakly equivalent as $\Sp$--algebras. Since $\alpha_2^p = 0$ in $\Gamma[\alpha_2]$, $\varphi(\sigma_2^p)=0$. This implies that the two non-weakly equivalent $H\Z$--algebras corresponding to $\sigma_2^p$ and $0$ are weakly equivalent as $\Sp$--algebras. These are the first example of non-trivial topological equivalences from \cite{dugger2007topological}. That is, there are two non-weakly equivalent $\Z$--DGAs with homology ring $\Lambda_{\F_p} (x_{2p-2})$ which are topologically equivalent. 

To complete the proof of Proposition \ref{classification 1}, we need to prove the following lemma. Let $X$ denote an $H\Z$--algebra or an $\Sp$--algebra with homotopy ring $\pi_* (X) \cong \Lambda_{\F_p} (x_n)$ for $\lvert x_n \rvert >0$. 


\begin{lemma} \label{lemma unique map}
Let $X$ be an $H\Z$--algebra as above, we have 
\[\pi_{0} \mathrm{map}_{H\Z \mhyphen alg}(X,H\F_p) = {pt}.\]
For an $\Sp$--algebra $X$ as above, we have
\[\pi_{0} \mathrm{map}_{\Sp \mhyphen alg}(X,H\F_p) = {pt}.\]

\end{lemma}

\begin{proof}
Hopkins--Miller obstruction theory states that obstructions to lifting a morphism in $E_2^{0,0} = \mathrm{Hom}_{\F_p \mhyphen alg} (\pi_* H\F_p \wedge_{H\Z} X, \F_p)$ to a map of $H\Z$--algebras lie in  the Andr{\'e}--Quillen cohomology for associative algebras, $E_2^{t,t-1}= \mathrm{Der}^t(\pi_* H\F_p \wedge_{H\Z} X, \Omega^{t-1}\F_p)$ for $t \geq 2$ and obstructions to homotopy unqiueness of the lift lie in $E_2^{t,t}= \mathrm{Der}^t(\pi_* H\F_p \wedge_{H\Z} X, \Omega^{t}\F_p)$ for $t \geq 1$. The desuspension functor $\Omega^s$ is defined by $\Omega^s M_* = M_{*+s}$ for a graded module $M_*$. By the Künneth spectral sequence, $\pi_0 H\F_p \wedge_{H\Z} X= \F_p$. Therefore there is only a single map in $\mathrm{Hom}_{\F_p \mhyphen alg} (\pi_* H\F_p \wedge_{H\Z} X, \F_p)$ because these morphisms preserve the identity and the grading. To show that the obstructions to existence and uniqueness are zero, first, note that by the Künneth spectral sequence, $\pi_* H\F_p \wedge_{H\Z} X$ is connected and $\Omega^t \F_p$ is in negative degrees for $t>0$. Therefore, $\mathrm{Der}(F^{\bullet +1} (\pi_* H\F_p \wedge_{H\Z} X), \Omega^t \F_p) = 0$ for $t>0$ where $F$ denotes the free associative algebra functor. Therefore, the cohomology of this co-simplicial abelian group is also zero. This proves the desired result. 
The argument for homotopy class of maps in $\Sp$--algebras is similar, the only difference is that one uses $\pi_* H\F_p \wedge_{\Sp} X$ instead of $\pi_* H\F_p \wedge_{H\Z} X$.

\end{proof}

\end{document}